\documentclass[11pt,reqno]{amsart}
\usepackage{fullpage}
\usepackage{dsfont}
\usepackage{amsmath}
\usepackage{amsfonts}
\usepackage{amssymb}
\usepackage{graphicx}
\usepackage{mathrsfs}
\usepackage{epsfig}
\usepackage{amsthm}
\usepackage{verbatim}
\usepackage{graphicx}
\usepackage{color}
\usepackage{url}

\usepackage[colorlinks,citecolor=black,linkcolor=black,
            bookmarksopen,
            bookmarksnumbered
           ]{hyperref}

\newcommand{\beq}{\begin{equation}}
\newcommand{\eeq}{\end{equation}}

\newcommand{\DL}{D_L}
\newcommand{\Lr}{\Lambda^r}

\newcommand{\dm}{ \partial_x m_Q}

\newcommand{\Le}{\partial_x m_L(\sqrt{\varepsilon}\partial_x) }

\newcommand{\LeQ}{\partial_x m_Q(\sqrt{\varepsilon}\partial_x) }

\newtheorem{theorem}{Theorem}[section]

\newtheorem{lemma}[theorem]{Lemma}

\usepackage{scalerel,stackengine}
\stackMath
\newcommand\reallywidehat[1]{%
\savestack{\tmpbox}{\stretchto{%
  \scaleto{%
    \scalerel*[\widthof{\ensuremath{#1}}]{\kern-.5pt\bigwedge\kern-.5pt}%
    {\rule[-\textheight/2]{1ex}{\textheight}}
  }{\textheight}%
}{0.5ex}}%
\stackon[1pt]{#1}{\tmpbox}%
}

\allowdisplaybreaks[3]

\author{María Cabrera Calvo}
\address{LJLL (UMR 7598), Sorbonne Universit\'e, UPMC, 4 place Jussieu, 75005, Paris, France (M. Cabrera Calvo)}
\email{cabreracalvo@ljll.math.upmc.fr}

\author{Fr\'ed\'eric Rousset}
\address{Universit\'e Paris-Saclay, CNRS,   Laboratoire de Math\'ematiques d'Orsay (UMR 8628),  91405 Orsay Cedex, France (F. Rousset)}
\email{frederic.rousset@universite-paris-saclay.fr}

\author{Katharina Schratz}
\address{LJLL (UMR 7598), Sorbonne Universit\'e, UPMC, 4 place Jussieu, 75005, Paris, France (K. Schratz)}
\email{katharina.schratz@ljll.math.upmc.fr}
 
\begin{document}
\begin{abstract}
We introduce  a  novel class of time integrators for dispersive equations which allow us to reproduce the dynamics of the  solution from the classical $ \varepsilon = 1$ up to long wave {limit} regime $
\varepsilon \ll 1 
$ on the  natural time scale  of the PDE $t= \mathcal{O}(\frac{1}{\varepsilon})$. Most notably our new schemes converge with rates at order $\tau \varepsilon$    over long times  $t=  \frac{1}{\varepsilon}$.  
\end{abstract}

\title[]{Time integrators for dispersive equations \\in the long wave regime}

\maketitle

\section{Introduction}
As a model problem we  consider
\begin{gather}
\label{BBM}
\partial_t u(t,x) +\partial_x m_L(\sqrt{\varepsilon}\partial_x) u(t,x)+ \varepsilon\partial_x m_Q(\sqrt{\varepsilon}\partial_x) u^2(t,x)  =0 \quad (t,x) \in \mathbb{R}\times  \mathbb{T} 
\end{gather}
with smooth symbols $m_L, m_Q$ satisfying  for $\xi  \in \mathbb{R}$
\begin{equation}\label{ass1}
\begin{aligned}
& m_L(i\xi) \in \mathbb{R}, \quad m_L(i\xi) = m_L(-i\xi),\quad  m_Q(i\xi) \in \mathbb{R}, \quad m_Q(i\xi) = m_Q(-i\xi),\\
&\left \vert m_L^{(4)}(i\xi)\right\vert \leq \frac{c_L}{1+\vert \xi\vert^{\beta_L}}
, \quad  \left \vert m_Q(i\xi)\right\vert \leq \frac{1}{1+\vert \xi\vert^{\beta_Q}}, \quad 
 \left \vert m_Q'(i\xi)\right\vert \leq \frac{1}{1+\vert \xi\vert}
\end{aligned}
\end{equation}
for some $\beta_L, \beta_Q \geq 0$. The class of equations \eqref{BBM} includes a large variety of models such as  the Benjamin--Bona--Mahony  (BBM) equation 
\begin{equation}\label{bbm}
m_L(i\xi)= m_Q(i\xi) = \frac1{1+\xi^2},
\end{equation}
 the  Korteweg--de Vries (KdV) equation 
\begin{equation}\label{kdv}
m_L(i \xi)= 1 - \xi^2, \quad m_Q(i\xi) = 1
\end{equation} and the Whitham equation
\begin{equation*}\label{kdv}
m_L(i\xi)= \sqrt{\frac{\tanh(\xi)}{\xi}}, \quad m_Q(i\xi) = 1.
\end{equation*} 
The model \eqref{BBM} can be rigorously derived in the long wave regime from many physical models including water waves, plasma, etc., see, e.g.,   \cite{ASL,ChR,Craig,GErR,Guo}. In particular, rigorous error estimates between the solution of \eqref{BBM} and the solution of the original model are established on the natural time scale $t= \mathcal{O}(\frac{1}{\varepsilon})$.

In this paper we introduce a novel class of numerical integrators  for \eqref{BBM}   based on  the long wave  behaviour of the dispersion relation 
 \begin{equation}\label{wlim}
 \begin{aligned}
& i \xi  \left(1  -  \frac{ m_L^{(2)}(0)}{2}  \xi^2 \right)\quad+\quad\text{higher order terms} \quad \text{with}\quad  \xi =\sqrt{\varepsilon}  k, \, k \in \mathbb{Z}.
\end{aligned}
 \end{equation}
At first order the  {long wave limit preserving} (LWP)    scheme takes the form
\begin{equation}\label{scheme}
\begin{aligned}
u^{n+1} & = \mathrm{e}^{-\tau \partial_x m_L(\sqrt{\varepsilon}\partial_x) } \left[  u^n 
 - \frac{1}{3\alpha}  m_Q(\sqrt{\varepsilon}\partial_x)  \Big(e^{\tau\alpha \varepsilon\partial_x^3}   \Big( e^{-\tau\alpha \varepsilon\partial_x^3}   \partial_x^{-1}       u^n \Big)^2 - \Big(   \partial_x^{-1}     u^n\Big)^2+  2 \varepsilon \tau \widehat{u^n_0} \partial_x u^n\Big)
 \right]
\end{aligned}
\end{equation}
where we have set $\alpha =  \frac{ m_L^{(2)}(0)}{2}$. 
Details on its construction will be given in Section~\ref{sec:dev1}. The  scheme~\eqref{scheme} (and its second order counterpart, see \eqref{schema2}) will allow us to reproduce the dynamics of the  solution  $u(t,x)$ of \eqref{BBM}   up to   long wave regimes $
\varepsilon \ll 1 
$ on  the natural long time scale $t= \mathcal{O}(\frac{1}{\varepsilon})$. More precisely, at first ($\sigma =1$) and second-order ($\sigma =2$)  we will   establish the  global error estimates
$$
\Vert u(t_n)- u^n \Vert_{L^2}  \leq \tau^\sigma  t_n \varepsilon^{2}c_0 e^{c_1 t_n  {\varepsilon}} \quad \text{on long time scales}\quad  t_n \leq \frac{1}{\varepsilon},\quad \sigma =1,2
$$
where  $c_0, c_1$ depend on certain Sobolev norms of $u$ (depending on $\beta_L$ and $\beta_Q$). We refer to Theorem~\ref{thm:glob1}  and Theorem \ref{thm:glob2} for the precise error estimates.  Note that the time scale  $t= \mathcal{O}(\frac{1}{\varepsilon})$ is also the natural time scale on the continuous level, i.e., for the PDE \eqref{BBM} itself. 

Compared to classical schemes, e.g., splitting or exponential integrator methods, our long wave limit preserving integrators  in particular
\begin{itemize}
\item allow for  approximations on large natural time scales $t= \mathcal{O}(\frac{1}{\varepsilon})$
\item converge with  rates at order $\tau^\sigma \varepsilon^2 t$.
\end{itemize}
Surprisingly, we can even achieve convergence of order $\tau \varepsilon$, i.e., a gain in $\varepsilon$,  over long times  $t=  \frac{1}{\varepsilon}$.  

 For the analysis of long-time energy conservation for Hamiltonian  partial differential equation  with the aid of Modulated Fourier Expansion and Birkhoff normal forms we refer to \cite{HLW,HLW1,HLW2,FGP,FGP1,FGP2} and the references therein.  Here we in contrast prove long time error estimates on the solution itself. In case of the nonlinear {Klein}-{Gordon} equation with weak nonlinearity $\varepsilon^2 u^3$ long time error estimates of splitting methods were recently established in \cite{Bao}. 
 
 The main challenge in the  theoretical and numerical analysis of \eqref{BBM} on   long time scales $t= \mathcal{O}(\frac{1}{\varepsilon})$ lies in  the  loss of derivative in the nonlinearity. This loss of derivative is clearly seen in case of the KdV equation \eqref{kdv} for which we face a  Burger's type nonlinearity $ \varepsilon\partial_x u^2 $. However, even in case of the BBM equation \eqref{bbm}   where we expect some regularisation through the structure of the leading operators (note that $\beta_L = \beta_Q =2$),  the smoothing only holds  with loss in~$\varepsilon$
 \begin{equation}\label{regi2}
\left\Vert  \varepsilon  \LeQ u^2 \right\Vert_r \leq \text{min}\left( {\sqrt{\varepsilon}} \Vert u^2 \Vert_{r},2 \Vert u \partial_x u \Vert_{r}\right).
\end{equation}
For BBM this may allow first order error estimates  at order $\tau \sqrt{\varepsilon} $   for classical splitting or exponential integrator methods up to  time $t= \mathcal{O}(\frac{1}{\sqrt{\varepsilon}})$, but   not  on the natural time scale of the PDE that is $t= \mathcal{O}(\frac{1}{\varepsilon})$. Our  new long wave limit adapted  discretisation \eqref{scheme} in contrast allows for    long time error estimates at order $\tau \varepsilon $ on the natural time scale  $t= \mathcal{O}(\frac{1}{{\varepsilon}})$.

In case of the BBM equation with a regularising nonlinearity ($\beta_L = \beta_Q =2$) we can thanks to the estimate  \eqref{regi2} play with the gain in $\varepsilon$ and loss of derivatives. This will allow  error bounds also for non smooth solutions, however, only on time scales $t=1$. More precisely, one could prove first-order convergence  in $H^r$ for solutions in $H^r$ ($r>1/2$), i.e., without any loss of derivative, for short times  $t=1$ at the cost of no longer gaining in~$\varepsilon$. Such low regularity estimates on short time scales without gain in $\varepsilon$ also hold true for classical   schemes, see for instance  \cite{CS21} for the analysis in case of  splitting discretisations. 
%

Our  idea for LWP schemes can be extended to higher order. We will give details on the second order integrator on long time scales in Section \ref{sec:scheme2}. Note that for the classical KdV equation (that is $\varepsilon = 1$ and without  transport term $\partial_x$), and nonlinear Schr\"odinger equations resonance based schemes were recently introduced in  \cite{HoS16,OS18} and (short time) error estimates for time $t=1$ were proven.  We also refer to \cite{H1,H2,Clem,OSu} for splitting, finite difference and Lawson-type methods  for the classical KdV equation on  time scales $t=1$.
  \\

\noindent{\bf Outline of the paper.}  In Section \ref{sec:scheme1} and  Section \ref{sec:scheme2} we introduce the  first- and second-order LWP scheme and carry out their convergence analysis over long times $t=\mathcal{O}\left( \frac{1}{\varepsilon}\right)$. Numerical experiments in Section \ref{sec:num} underline our theoretical findings.\\

\noindent{\bf Notation and assumptions.} In the following we will assume that $m_L^{(2)}(0)=2$ which implies  (as $ \alpha = m_L^{(2)}(0)/2 $) that $ \alpha =1$ in \eqref{wlim}. Our analysis also holds true for any $ m_L^{(2)}(0)  \in \mathbb{R}$.  For practical implementation issues we will impose periodic boundary conditions that is  $x \in \mathbb{T} =[-\pi ,\pi]$. Our result can be extended to the full space $x\in \mathbb{R}$. We denote by $(\cdot,\cdot)$ the standard scalar product $
(f,g) = \int_{\mathbb{T}} f g dx
$  and by $\|.\|_r$ the standard $H^r(\mathbb{T})$ norm. In particular, for $r>1/2$, we will exploit the standard bilinear estimate
\begin{gather}
\label{bil_est}
\|fg\|_r \leq C_r \|f\|_r \|g\|_r.
\end{gather}
 For $v(x) = \sum_{k \in \mathbb{Z}} \hat v_k e^{i k x}$ we set $\partial_x^{-1} v(x) := \sum_{k\neq 0} \hat v_k e^{i k x}$. 
Let us also define $\Lambda = (1+\vert \xi \vert^2)^{\frac12}.$

\section{A first-order long wave limit preserving scheme}\label{sec:scheme1}
In a first section we will formally derive the  LWP scheme \eqref{scheme} (see Section \ref{sec:dev1}). Then we will carry out its convergence analysis and establish  long time error estimate  (see Section \ref{sec:err1}).
\subsection{Derivation of the scheme}\label{sec:dev1}
Recall  Duhamel's formula  of  \eqref{BBM} 
\begin{equation*}\label{expu1}
\begin{aligned}
u(t) &   = e^{- t \Le} u(0)
- \varepsilon \LeQ e^{-t \Le}   \int_0^t e^{s \Le} u^2(s)ds.
\end{aligned}
\end{equation*}
Iterating the above formula, i.e., using that 
\begin{equation*}\label{expu1}
\begin{aligned}
u(s) &   = e^{- s \Le} u(t_n)
+\mathcal{O}\left( s \varepsilon \LeQ  u^2\right)
\end{aligned}
\end{equation*}
we see that formally
\begin{equation*}\begin{aligned}
u(t) &   
 \approx e^{- t\Le} \left[u(0)
- \varepsilon \LeQ   \int_0^t e^{s \Le} \left( e^{- s \Le}  u(0)\right)^2 ds\right].
\end{aligned}
\end{equation*}
The key point lies in embedding the long wave limit behaviour  (cf. \eqref{wlim})
$$
\DL = \Le - (\partial_x + \varepsilon \partial_x^3) 
= \mathcal{O}\left( {\varepsilon^{2 } \partial_x^{5 }}   m_L^{(4)}(\sqrt{\varepsilon}\partial_x) \right)
$$
 into our numerical discretisation. This motivates (for sufficiently smooth solutions) the following approximation
\begin{equation*}\begin{aligned}
u(t) &   
 \approx   e^{-t \Le} \left[u(0)
- \varepsilon \LeQ   \int_0^t  e^{s(\partial_x + \varepsilon \partial_x^3) } \left( e^{- s (\partial_x + \varepsilon \partial_x^3) }  u(0)\right)^2 ds\right].
\end{aligned}
\end{equation*}
We may solve the oscillatory integral by the observation that
$$
\varepsilon\partial_x \int_0^t  e^{s(\partial_x + \varepsilon \partial_x^3) } \left( e^{- s (\partial_x + \varepsilon \partial_x^3) } v \right)^2 ds = \frac{1}{3}   \mathrm{e}^{t   \varepsilon \partial_x^3}   \left[ \mathrm{e}^{ -t \varepsilon \partial_x^3  }  (\partial_x^{-1} v)^2\right]   -\frac13 (\partial_x^{-1} v)^2+ 2 \varepsilon t \hat{v}_0 \partial_x v,
$$
see \eqref{resi}. Based on the long wave limit behaviour we thus find the following approximation
\begin{equation*}\begin{aligned}
u(t) &   
 \approx   e^{-t \Le} \left[u(0)
- m_Q(\sqrt{\varepsilon}\partial_x) \left(
 \frac{1}{3}   \mathrm{e}^{t   \varepsilon \partial_x^3}   \left( \mathrm{e}^{ -t \varepsilon \partial_x^3  }  (\partial_x^{-1} u(0))^2\right)   -\frac13 (\partial_x^{-1} v)^2+ 2 \varepsilon t \widehat{u(0)}_0 \partial_x u(0)\right)
\right]
\end{aligned}
\end{equation*}
which builds the basis of our LWP  scheme \eqref{scheme}.
\subsection{Error analysis}\label{sec:err1}
In this section we carry out the error analysis of the filtered scheme \eqref{scheme} over long times $t=\mathcal{O}\left( \frac{1}{\varepsilon}\right)$ We start with the local error analysis. For this purpose we will denote by $\varphi^t$ the exact flow of   \eqref{BBM} and by $\Phi^\tau$ the numerical flow defined by the scheme \eqref{scheme}, such that
$$
u(t_n+\tau) = \varphi^\tau(u(t_n) ) \quad \text{and}\quad u^{n+1} = \Phi^\tau(u^n).
$$
\subsubsection{Local error analysis}
We will exploit the following estimate which regularises for $\beta_Q > 1/2$.
\begin{lemma}
\label{lemma_reg}
Let $f\in H^{r+1-\beta_Q}(\mathbb{T})$. It holds that
\begin{equation*}
\|\varepsilon\LeQ f\|_r \leq \varepsilon^{1-\beta_Q} \|f\|_{r+1-\beta_Q}.
\end{equation*}
\end{lemma}
\begin{proof}
The assertion follows thanks to the estimate
\begin{align*}
\|\varepsilon\LeQ f\|_r^2 & =  \sum_{k\in\mathbb{Z}} (1 + |k|)^{2r} \bigg| \frac{\varepsilon ik}{1+(\sqrt{\varepsilon} k)^{\beta_Q}} \bigg|^2 |\hat{f}_k|^2
 \leq \varepsilon^{2(1-\beta_Q)} \|f\|_{r+1-\beta_Q }^2.
\end{align*}
\end{proof}
\begin{lemma} \label{thm:loc} Fix $r>1/2$. Then, the local error  $ \varphi^\tau(u(t_n))- \Phi^\tau(u(t_n))$ satisfies for $$ \beta :=min(2, \beta_L + \beta_Q)$$
the estimate
\begin{align*}
\Vert \varphi^\tau(u(t_n))- \Phi^\tau(u(t_n)) \Vert_r &  \leq \tau^2  \varepsilon^{2}  c\left(\sup_{t_n\leq t\leq t_{n+1}} \Vert u(t) \Vert_{r+2}\right)  + c_L \tau^{2} \varepsilon^{3-\frac{\beta}2} c\left(\sup_{t_n\leq t\leq t_{n+1}} \Vert u(t) \Vert_{r+6- \beta }\right).\end{align*}
\end{lemma}
\begin{proof}
Iterating Duhamel's formula  of  \eqref{BBM} yields that
\begin{equation}\label{expu1}
\begin{aligned}
u(t_n+\tau) &   = e^{- \tau \Le} u(t_n)
- \varepsilon \LeQ e^{- \tau \Le}   \int_0^\tau e^{s \Le} u^2(t_n+s)ds\\
& =  e^{- \tau \Le} u(t_n)
\\&- \varepsilon \LeQ e^{- \tau \Le}   \int_0^\tau e^{s \Le} \left( e^{- s \Le}  u(t_n)\right)^2 ds\\& + \mathcal{R}_1(\varepsilon, \tau, u )
\end{aligned}
\end{equation}
with the remainder 
\begin{equation}
\mathcal{R}_1(\varepsilon, \tau, u )
 =  \varepsilon \LeQ e^{- \tau \Le}   \int_0^\tau e^{s \Le} \Big[ \left( e^{- s \Le}  u(t_n)\right)^2- u^2(t_n+s)\Big] ds.
\end{equation}
Thanks to the observation that
$$
u(t_n+s)    = e^{- s\Le} u(t_n)
- \varepsilon \LeQ e^{- s \Le}   \int_0^s e^{s_1 \Le} u^2(t_n+s_1)ds_1
$$
the remainder $\mathcal{R}_1(\varepsilon, \tau, u )$ is of the following form
$$
\mathcal{R}_1(\varepsilon, \tau, u ) = \mathcal{O}\left(\tau^2  \varepsilon \LeQ \left(u(t)  \varepsilon \LeQ \left( u^2(t)\right)  \right)  \right).
$$
Thanks to assumption \eqref{ass1}  (which guarantees the boundedness of the symbol $m_Q$) we can thus conclude that
\begin{equation}\label{R1}
\Vert \mathcal{R}_1(\varepsilon, \tau, u )
 \Vert_r \leq  \tau^2 \varepsilon^{2 }c\left(\sup_{t_n\leq t\leq t_{n+1}} \Vert u(t) \Vert_{r+2}\right).
\end{equation}
Taylor series expansion of the symbol $m_L(\delta)$ around $\delta = 0$ yields that
\begin{align*}
m_L(\delta)  & =m_L(0)  +   \xi m_L'(0) +    \frac{\delta^2}{2} m_L^{''}(0) +  \frac{\delta^3}{3!} m_L^{(3)}(0) + \int_0^\delta \frac{(\delta-\tilde \delta)^3}{3!} m_L^{(4)}(\tilde \delta) d \tilde \delta \\
 & =  1  +    {\delta^2}  m_L^{''}(0) + \int_0^\delta\frac{(\delta-\tilde \delta)^3}{3!} m_L^{(4)}(\tilde \delta) d \tilde \delta\\
\end{align*}
where in the last step we have used  the assumptions \eqref{ass1} (which implies that $m^{(2\ell +1)}(0)=0$) and the assumption that (without loss of generality) $m_L^{(2)}(0) =2$. Together with the assumption that
$ \left \vert m_L^{(4)}(i\xi)\right\vert \leq \frac{c_L}{1+\vert \xi\vert^{\beta_L}}$ (see again  \eqref{ass1})
we thus find that
 \begin{align}\label{opexp}
\DL = \Le - (\partial_x + \varepsilon \partial_x^3) 
= \mathcal{O}\left( c_L\frac{\varepsilon^{2 } \partial_x^{5 }}{1+ (\sqrt{\varepsilon}\vert \partial_x\vert)^{\beta_L}}   \right).
\end{align}
This  allows the following expansion of the oscillations
\begin{equation}\label{osc}
e^{ \pm s  \Le}  = 
e^{  \pm s   (\partial_x + \varepsilon \partial_x^3)} +\mathcal{O}\left( c_L\frac{\varepsilon^{2 } \partial_x^{5 }}{1+ (\sqrt{\varepsilon}\vert \partial_x\vert)^{\beta_L}}   \right).
\end{equation}
Employing these expansion for the oscillations  in the remaining integral term in  \eqref{expu1} 
 yields together with Lemma \ref{lemma_reg} that
\begin{equation}\label{expu2}
\begin{aligned}
u(t_n+\tau)   
& =  e^{- \tau \Le} u(t_n)
- \varepsilon \LeQ e^{- \tau \Le}   \int_0^\tau e^{ s  (\partial_x + \varepsilon \partial_x^3)} \left( e^{- s (\partial_x + \varepsilon \partial_x^3)}  u(t_n)\right)^2 ds \\&+\mathcal{R}_2(\varepsilon, \tau, u )
\end{aligned}
\end{equation}
where the remainder $\mathcal{R}_2 (\varepsilon, \tau, u ) $ is thanks to \eqref{ass1} of type
$$  \mathcal{O}\left(c_L \tau^{2}  \frac{\varepsilon^{2 } \partial_x^{5 }}{1+ (\sqrt{\varepsilon}\vert \partial_x\vert)^{\beta_L}}  \varepsilon \LeQ u^2 \right) =   \mathcal{O}\left(c_L \tau^{2}  \frac{\varepsilon^{2 } \partial_x^{5 }}{1+ (\sqrt{\varepsilon}\vert \partial_x\vert)^{\beta_L}}  \frac{\varepsilon  \partial_x }{1+ (\sqrt{\varepsilon}\vert \partial_x\vert)^{\beta_Q}} u^2 \right)
$$
such that
\begin{equation}\label{R2}
\Vert \mathcal{R}_2(\varepsilon, \tau, u )
 \Vert_r \leq  c_L \tau^{2} \varepsilon^{3-\frac{\beta}2} c\left(\sup_{t_n\leq t\leq t_{n+1}} \Vert u(t) \Vert_{r+6- \beta }\right), \quad \beta :=min(2, \beta_L + \beta_Q).
\end{equation}
Next we calculate with the aid of the Fourier transform $v(x) = \sum_{k \in \mathbb{Z}} \hat v_k e^{i k x}$ and the definition $$\partial_x^{-1} v(x) = \sum_{k\neq 0} \hat v_k e^{i k x}$$  that
\begin{equation}
\begin{aligned}\label{resi}
\mathcal{I}(\tau,\varepsilon, v) &= \varepsilon \partial_x \int_0^\tau \mathrm{e}^{ s  \varepsilon \partial_x^3  }   \left[ \mathrm{e}^{ -s  \varepsilon \partial_x^3 }   v^2\right]  ds\\
& = \varepsilon \sum_{\substack{\ell+m = k\\ \ell, m \neq 0}}e^{ik x} \hat v_\ell \hat v_m (ik)  \int_0^\tau e^{- 3i s \varepsilon k \ell m } ds + 2 \varepsilon \tau  \hat{v}_0 \partial_x v\\
& =  \sum_{\substack{\ell+m = k\\ \ell, m \neq 0}}e^{ik x} \hat v_\ell \hat v_m (ik)  \left( e^{- 3i \tau  \varepsilon k \ell m } -1 \right)
\frac{1}{3(i\ell)(im)} + 2 \varepsilon \tau  \hat{v}_0 \partial_x v\\
& = \frac{1}{3}   \mathrm{e}^{ \tau   \varepsilon \partial_x^3}   \left[ \mathrm{e}^{ -\tau  \varepsilon \partial_x^3  }  (\partial_x^{-1} v)^2\right]   -\frac13 (\partial_x^{-1} v)^2+ 2 \varepsilon \tau  \hat{v}_0 \partial_x v,
\end{aligned}
\end{equation}
see also \cite{HoS16} in case of no advection term $\partial_x$. Plugging the above relation into \eqref{expu2} we obtain
\begin{align*}
 \varphi^\tau(u(t_n))= \Phi^\tau(u(t_n)) +  \sum_{i=1,2}\mathcal{R}_i(\varepsilon, \tau, u ),
 \end{align*}
where $\mathcal{R}_1$ satisfies \eqref{R1} and $\mathcal{R}_2$ satisfies \eqref{R2}. This concludes the proof.
\end{proof}
\subsubsection{Stability analysis}
In order to carry out the stability analysis we need the following Lemma.
\begin{lemma} \label{lem:lemi}
For $\vert m_Q(i\xi) \vert \leq 1$ and $\vert m_Q'(i\xi) \vert \leq \frac{1}{1+\vert \xi \vert}$ it holds that
$$
 \Vert  [m_Q(\sqrt{\varepsilon}\partial_x),w]   \partial_x \Lr v \Vert_{L^2}  \leq \Vert w \Vert_{r+1} \Vert v \Vert_{H^r}.
$$
\end{lemma}
\begin{proof}
Recall that $\Lambda = (1+\vert \xi \vert^2)^{\frac12}$. The $k$-th Fourier coefficient of $ [m_Q,w]   \partial_x \Lr v $ is given by
\begin{align}\label{KF}
\reallywidehat{[m_Q(\sqrt{\varepsilon}\partial_x),w]   \partial_x \Lr v}(k) &= \sum_{l} \hat{w}(k-l)  \Big[m_Q(\sqrt{\varepsilon}i l) - m_Q(\sqrt{\varepsilon} ik) \Big]  i l (1+l^2)^\frac{r}{2} \hat{v}(l).
\end{align}
Newt we note that as $\vert m_Q'(i\xi) \vert \leq \frac{1}{1+\vert \xi \vert}$ we have\\
(i) if $\vert l \vert \leq 2 \vert k - l\vert$ that $\vert m_Q(\sqrt{\varepsilon}i l) - m_Q(\sqrt{\varepsilon} ik) \vert \leq 2$\\
(ii) if $\vert l \vert > 2 \vert k - l\vert$ that 
$$
\vert m_Q(\sqrt{\varepsilon} i l) - m_Q(\sqrt{\varepsilon}i k) \vert  \leq \sqrt{\varepsilon} \vert k - l \vert \int_0^1 \frac{1}{ 1 + \sqrt{\varepsilon}  \vert l + s (k-l)\vert} ds \leq \frac{\sqrt{\varepsilon} \vert k - l\vert}{ 1+ \sqrt{\varepsilon}\vert l\vert} \leq \frac{\vert k - l\vert}{\vert l \vert}.
$$
Hence we can conclude that
$$
\vert m_Q(\sqrt{\varepsilon} i l) - m_Q(\sqrt{\varepsilon}i  k)\vert  \vert l \vert \leq \vert k - l \vert.
$$
Plugging the above estimate into \eqref{KF} we obtain that
\begin{align*}
\left \vert \reallywidehat{[m_Q(\sqrt{\varepsilon}\partial_x),w]   \partial_x \Lr v}(k)\right \vert \leq  \sum_{l}  \vert k - l \vert \vert \hat{w}(k-l)  \vert \vert (1+l^2)^\frac{r}{2}\vert \vert \hat{v}(l)\vert
\end{align*}
which implies the assertion.

\end{proof}
\begin{lemma} \label{thm:stab} Fix  $r\geq 1$. The numerical flow defined by the scheme \eqref{scheme} is ${\varepsilon}$-stable in $H^r$ in the sense that for two functions $w\in H^{r+1}$ and $v\in H^r$  we have that 
$$
\Vert \Phi^\tau(w)- \Phi^\tau(v) \Vert_r  \leq  e^{\tau {\varepsilon} B} \Vert w-v\Vert_r, \qquad B = B(\Vert w\Vert_{r+1}, \Vert v \Vert_r),
$$
where the constant $B$ depends on the $H^{r+1}$ norm of $w$ and  $H^r$ norm of $v$.
\end{lemma}
\begin{proof} Fix $r\geq 1$.  For the stability analysis we will rewrite the numerical flow back  \eqref{scheme}  in its integral form. Tanks to \eqref{resi} we observe that
$$
 \Phi^\tau(v) =\mathrm{e}^{-\tau \Le }  v   -  \varepsilon\LeQ \mathrm{e}^{-\tau \Le } \int_0^\tau \mathrm{e}^{ s   \varepsilon \partial_x^3 }   \left[ \mathrm{e}^{ -s   \varepsilon \partial_x^3}   v^2\right] ds.
$$
We need to show that
\begin{align}\label{todo}
\left \vert \left( \Lr\dm (w v), \Lr v \right)\right\vert \leq \Vert w \Vert_{r+1} \Vert v \Vert^2_r
\end{align}
where for shortness we write $m_Q =m_Q(\sqrt{\varepsilon}\partial_x)$.

Let us note that
$$
 \Lr\dm (w v) = \Lr m_Q (w \partial_x v) + \Lr m_Q (v \partial_x w),
$$
where thanks to the boundedness of $m_Q$ (see \eqref{ass1}) we have that
$$
\left \vert \left( \Lr m_Q (v \partial_x w),  \Lr v \right)\right\vert \leq
\Vert v \Vert_r^2 \Vert  v \partial_x w  \Vert_{r}
\leq  \Vert v \Vert_r^2 \Vert w \Vert_{r+1}.
$$
Thus we obtain that
\begin{align}\label{todo2}
\left \vert \left( \Lr\dm (w v), \Lr v \right)\right\vert  \leq \left \vert \left( \Lr m_Q (w \partial_x v), \Lr v \right)\right\vert  +  \Vert v \Vert_r^2 \Vert w \Vert_{r+1}
\end{align}
and it remains to derive a suitable bound on $\left \vert \left( \Lr m_Q (w \partial_x v), \Lr v \right)\right\vert $. For this purpose let us note that
\begin{equation}\label{K1}
\Lr m_Q (w \partial_x v) = m_Q( w \Lr \partial_x v) + m_Q ( [\Lr, w] \partial_x v).
\end{equation}
For the second term in \eqref{K1} we see that
\begin{equation}\label{K12}
\begin{aligned}
\left \vert \left( m_Q ( [\Lr, w] \partial_x v),   \Lr v \right)\right\vert\leq \Vert [\Lr, w] \partial_x v\Vert_{L^2} \Vert v \Vert_r  \leq \left( \Vert \partial_x v\Vert_{L^2} \Vert w \Vert_{r+1} + \Vert w \Vert_{L^\infty} \Vert v \Vert_r\right) \Vert v\Vert_r.
\end{aligned}
\end{equation}
For first term in \eqref{K1} we see that as $m_Q = m_Q^\ast$
\begin{equation*}\label{K11}
\begin{aligned}
 \left( m_Q( w \Lr \partial_x v),   \Lr v \right) & =  \left(  w \Lr \partial_x v, m_Q  \Lr v \right)  = -  \left(  (\partial_x w) \Lr  v, m_Q  \Lr v \right) - \left(  w \Lr  v, m_Q   \partial_x \Lr v \right)\\
 & =   -  \left(  (\partial_x w) \Lr  v, m_Q  \Lr v \right) - \left(   \Lr  v, m_Q (w   \partial_x \Lr v) \right)
 -  \left(   \Lr  v, [m_Q,w]   \partial_x \Lr v \right).
 \end{aligned}
\end{equation*}
Hence,
$$
 \left( m_Q( w \Lr \partial_x v),   \Lr v \right)  = - \frac12 \left(  (\partial_x w) \Lr  v, m_Q  \Lr v \right) 
 - \frac12   \left(   \Lr  v, [m_Q,w]   \partial_x \Lr v \right)
$$
which implies thanks to Lemma \ref{lem:lemi} and the assumptions \eqref{ass1} on $m_Q$ that
\begin{equation}\label{K11}
\begin{aligned}
\left \vert \left( m_Q( w \Lr \partial_x v),   \Lr v \right)  \right\vert \leq
\Vert \partial_x w\Vert_{L^\infty} \Vert v \Vert_r^2 + \Vert v \Vert_r \Vert  [m_Q,w]   \partial_x \Lr v \Vert_{L^2}  \leq c  \Vert v \Vert_r^2 \Vert w \Vert_{r+1}.
\end{aligned}
\end{equation}
Plugging \eqref{K12} and \eqref{K11} into \eqref{K1} wields that
$$
\left \vert \left( \Lr m_Q (w \partial_x v), \Lr v \right)\right\vert  \leq c  \Vert v \Vert_r^2 \Vert w \Vert_{r+1}
$$
which by \eqref{todo2} implies the desired esitmate \eqref{todo}.

\end{proof}

\subsubsection{Global error estimate} 
\begin{theorem} \label{thm:glob1} Fix $ \beta :=min(2, \beta_L + \beta_Q)$ and assume that the solution $u$ of \eqref{BBM} satisfies $u\in \mathcal{C}([0,T];H^{6- \beta})$. Then there exists a $\tau_0>0$ such that for all $0<\tau \leq \tau_0$  the global error estimate holds for  $u^n$ defined in \eqref{scheme}
\begin{align*}
\Vert u(t_n)- u^n \Vert_{L^2} &  \leq   t_n   \tau  \varepsilon^{2}  c\left(\sup_{0\leq t\leq t_{n}} \Vert u(t) \Vert_{2}\right)  e^{c t_n  {\varepsilon}}+ c_L   t_n     \tau \varepsilon^2 \varepsilon^{1-\frac{\beta}2} c\left(\sup_{0\leq t\leq t_{n}} \Vert u(t) \Vert_{6- \beta }\right)   e^{c t_n  {\varepsilon}}
\end{align*}
where $c$ depends on the $H^{2}$ norm of the solution $u$.
\end{theorem}
\begin{proof}
The assertion in $H^r$, $r\geq 1$, follows by the local error bound given in Lemma \ref{thm:loc} together with the stability estimate in Lemma \ref{thm:stab} (with the stronger norm placed on the exact solution $u(t_n)$) via a Lady Windermere's fan argument (see, e.g., \cite{HLW}). Then under the given regularity assumptions on the exact solution (which imply that $u$ is at least in $H^4$) we can prove the corresponding $L^2$ error bound by first proving convergence (with reduced order in $\tau$ but full gain of at least one factor $\varepsilon$) in $H^{1}$, i.e., 
\begin{align*}
\Vert u(t_n)- u^n \Vert_{H^1} &  \leq  t_n   \tau^\delta  \varepsilon   c\left(\sup_{0\leq t\leq t_{n}} \Vert u(t) \Vert_{4}\right)  e^{c t_n  {\varepsilon}}  
\end{align*}
for some $\delta = \delta(\beta_L,\beta_Q) >0$. Thanks to the estimate
$$
\Vert  u^n \Vert_{H^1} \leq \Vert u(t_n)- u^n \Vert_{H^1} + \Vert u(t_n) \Vert_{H^1} 
$$ this  will give us a priori the boundedness of the numerical solution in $H^{1}$ over long times $t_n=\frac{1}{\varepsilon}$. For details on the latter approach in case of short time ($t=1$) estimates for splitting methods for cubic Schr\"odinger we refer to \cite{Lubich08}.
\end{proof}

\section{A second-order long wave limit preserving scheme}\label{sec:scheme2}
Our second order LWP  scheme for \eqref{BBM} takes the form 
\begin{equation}\label{schema2}
\begin{aligned}
u^{n+1} &=   \mathrm{e}^{-\tau \partial_x m_L(\sqrt{\varepsilon}\partial_x) } \left[  u^n 
 - \frac{1}{3\alpha}  m_Q(\sqrt{\varepsilon}\partial_x)  \Big(e^{\tau\alpha\varepsilon\partial_x^3}   \Big( e^{-\tau\alpha\varepsilon\partial_x^3}   \partial_x^{-1}       u^n \Big)^2 - \Big(   \partial_x^{-1}     u^n\Big)^2+  2 \varepsilon \tau \widehat{u^n_0} \partial_x u^n\Big)
 \right]  \\& +  {\tau^2} \varepsilon^2\LeQ \Psi_{m_Q} \Big(u^n   \Psi_{m_Q}\LeQ (u^n u^n)\Big)  \\&
 -  \frac{\tau^2}{2}  \varepsilon \LeQ     \Psi_{\DL,m_Q}\DL   (u^n u^n)    + {\tau^2}     \varepsilon \LeQ  \Psi_{\DL,m_Q}\left( u^n  \DL u^n   \right) 
  \end{aligned}
\end{equation}
where we recall that $ \alpha = m_L^{(2)}(0)/2 $ and 
$$
\DL(\sqrt{\varepsilon} \partial_x) = \Le - (\partial_x + \alpha \varepsilon\partial_x^3).
$$
For stability issues  we have introduced the  filter functions $$\Psi_{m_Q}(\sqrt{\varepsilon}\partial_x) \quad \text{and}\quad   \Psi_{\DL,m_Q} = \Psi_{\DL,m_Q}(\sqrt{\varepsilon}\partial_x)$$ satisfying
\begin{equation}\label{filter}
\begin{aligned}
&\left \Vert \tau \Psi_{m_Q}(\sqrt{\varepsilon}\partial_x)\LeQ v\right \Vert_r \leq \Vert v \Vert_r , \quad 
\left \Vert\Psi_{m_Q }(\sqrt{\varepsilon}\partial_x) v -  v\right \Vert_r \leq \tau \Vert \LeQ v\Vert_{r}
\\
& \left \Vert \tau \LeQ   \DL   \Psi_{\DL,m_Q}(\sqrt{\varepsilon}\partial_x)   v\right \Vert_r \leq \Vert v \Vert_r , \,
\left \Vert \Psi_{\DL,m_Q}(\sqrt{\varepsilon}\partial_x)   v -  v\right \Vert_r \leq \tau \Vert  \LeQ\DL v\Vert_{r}.
\end{aligned}
\end{equation}
For an introduction to filter functions we refer to \cite{HLW}.\\

In a first section we will derive the LWP  scheme \eqref{schema2} (see Section \ref{sec:dev2}). Then we will carry out its long time error estimate (see Section \ref{sec:err2}). We will again assume without loss of generality that $ \alpha = 1$.

\subsection{Derivation of the scheme}\label{sec:dev2}
Iterating Duhamel's formula   \eqref{BBM} yields that
\begin{equation*}
\begin{aligned}
u(t_n+\tau) &   = e^{- \tau \Le} u(t_n)
- \varepsilon \LeQ e^{- \tau \Le}   \int_0^\tau e^{s \Le} u^2(t_n+s)ds\\
& =  e^{- \tau \Le} u(t_n)
\\& - \varepsilon \LeQ e^{- \tau \Le}   \int_0^\tau e^{s \Le} \Big( e^{- s \Le}  u(t_n)\\&\qquad -  \varepsilon \LeQ e^{- s \Le}   \int_0^s e^{s_1 \Le} u^2(t_n+s_1)ds_1\Big)^2 ds.
\end{aligned}
\end{equation*}
Employing the approximation
\begin{align*}
 \varepsilon \LeQ & e^{- s \Le}   \int_0^s e^{s_1 \Le} u^2(t_n+s_1)ds_1 \\ &= 
{s} \varepsilon \LeQ     u^2(t_n)  + \mathcal{O}(s^2 \varepsilon \Le \LeQ u^2)
\end{align*}
we obtain that
\begin{equation}\label{uexp2}
\begin{aligned}
u(t_n+\tau)    
& =  e^{- \tau \Le} u(t_n)
\\&- \varepsilon \LeQ e^{- \tau \Le}   \int_0^\tau e^{s \Le} \left( e^{- s \Le}  u(t_n)\right)^2 ds\\
& + \tau^2 \varepsilon \LeQ e^{- \tau \Le}    \left( u(t_n) \varepsilon \LeQ  u^2(t_n)\right)    + \mathcal{R}_1(\tau,\varepsilon,u). \end{aligned}
\end{equation}
The remainder $\mathcal{R}_1(\tau,\varepsilon,u)$ is thereby of order $$\mathcal{O}(s^2 \varepsilon^2 \Le \LeQ  \LeQ u^2)$$ which implies by  assumption \eqref{ass1} together with the observation  (see \eqref{opexp})
 $$
\DL = \Le - (\partial_x + \varepsilon \partial_x^3) =\mathcal{O}\left( c_L\frac{\varepsilon^{2 } \partial_x^{5 }}{1+ (\sqrt{\varepsilon}\vert \partial_x\vert)^{\beta_L}}   \right)
 $$
 the following bound
\begin{align}\label{2R1}
\Vert \mathcal{R}_1(\tau,\varepsilon,u)\Vert_r \leq \tau^3 \varepsilon^2 c \left(\sup_{t_n \leq t \leq t_{n+1}}  \Vert  u\Vert_{r+5}\right)+ \tau^3 \varepsilon^2 \varepsilon^{1-\beta_0} c \left(\sup_{t_n \leq t \leq t_{n+1}}  \Vert  u\Vert_{r+7-2\beta_0}\right)
\end{align}
with $\beta_0 = \text{min}(1, \beta_0)$.

Next we employ the following lemma.
\begin{lemma}\label{lem:osc2}
It holds that 
\begin{align*}
  \varepsilon &\LeQ e^{- \tau \Le}  \int_0^\tau e^{s \Le} \left( e^{- s \Le}  v\right)^2 ds\\&  = 
\varepsilon \LeQ e^{- \tau \Le}    \int_0^\tau e^{s  (\partial_x + \varepsilon \partial_x^3) } \left(e^{-s  (\partial_x + \varepsilon \partial_x^3) }  v\right)^2 ds\\
  & +  \frac{\tau^2}{2}  \varepsilon \LeQ   \Psi_{\DL}\DL  v^2   - {\tau^2}     \varepsilon \LeQ  \Psi_{\DL}  \left( v \DL  v   \right) + \mathcal{R}_2(\tau,\varepsilon,u) 
 \end{align*}
with the remainder
 \begin{align}\label{2R2}
\Vert \mathcal{R}_2(\tau,\varepsilon,u)\Vert_r \leq c_L\tau^3 \varepsilon^2 \varepsilon^{1-\beta_1/2} c \left(\sup_{t_n \leq t \leq t_{n+1}}  \Vert  u\Vert_{r+7-\beta_1}\right) +  c_L\tau^3   \varepsilon^2  \varepsilon^{3-\beta_2/2} c \left(\sup_{t_n \leq t \leq t_{n+1}}  \Vert  u\Vert_{r+11-\beta_2}\right) 
\end{align}
where $\beta_1 = \text{min}(2,\beta_Q+\beta_L)$ and $\beta_2 = \text{min}(6,\beta_Q+2 \beta_L)$.
\end{lemma}
\begin{proof}
Note that
\begin{align*}
 \int_0^\tau & e^{s \Le} \left( e^{- s \Le}  v\right)^2 ds
      \\&  = 
  \int_0^\tau \left( e^{s \Le}-  e^{s  (\partial_x + \varepsilon \partial_x^3) } \right) \left( e^{- s \Le}  v\right)^2 ds\\
  &+   \int_0^\tau  e^{s  (\partial_x + \varepsilon \partial_x^3) }  \Big[   \left( e^{- s \Le}  v\right)  \left( e^{-s \Le}-  e^{-s  (\partial_x + \varepsilon \partial_x^3) } \right)  v  \Big]ds\\
    &+   \int_0^\tau  e^{s  (\partial_x + \varepsilon \partial_x^3) }  \Big[  \left(  e^{-s  (\partial_x + \varepsilon \partial_x^3) }   v \right)     \left( e^{-s \Le}-  e^{-s  (\partial_x + \varepsilon \partial_x^3) } \right)    v  \Big]ds\\
    & + \int_0^\tau e^{s  (\partial_x + \varepsilon \partial_x^3) } \left(e^{-s  (\partial_x + \varepsilon \partial_x^3) }  v\right)^2 ds.
  \end{align*}
 Hence, using that (see \eqref{opexp})
 $$
\DL = \Le - (\partial_x + \varepsilon \partial_x^3) =\mathcal{O}\left( c_L\frac{\varepsilon^{2 } \partial_x^{5 }}{1+ (\sqrt{\varepsilon}\vert \partial_x\vert)^{\beta_L}}   \right)
 $$
 together with the expansions
 $$
  e^{-s \Le}= 1+\mathcal{O}(s\Le), \qquad  e^{\pm s  (\partial_x + \varepsilon \partial_x^3) } = 1 +\mathcal{O} \left( s  (\partial_x + \varepsilon \partial_x^3) \right) 
 $$
we obtain that
 \begin{align*}
 \int_0^\tau & e^{s \Le} \left( e^{- s \Le}  v\right)^2 ds\\&  = 
  \int_0^\tau e^{s  (\partial_x + \varepsilon \partial_x^3) } \left(e^{-s  (\partial_x + \varepsilon \partial_x^3) }  v\right)^2 ds + \frac{\tau^2}{2} \DL   v^2  - \tau^2      v  \DL v     \\&+\mathcal{O}\left(c_L \tau^3\frac{\varepsilon^{4 } \partial_x^{10 }}{1+ (\sqrt{\varepsilon}\vert \partial_x\vert)^{2\beta_L}}  p(v)\right)+ \mathcal{O}\left(\tau^3 c_L\frac{\varepsilon^{2 } \partial_x^{6 }}{1+ (\sqrt{\varepsilon}\vert \partial_x\vert)^{\beta_L}}  p(v) \right)
   \end{align*}
   for some polynomials  $p$ of degree 2 in $v$.
This yields together with the properties on the filter functions (cf. \eqref{filter}) 
$$
\Psi_{ m_Q }(\sqrt{\varepsilon} \partial_x) = 1 +\mathcal{O}(\tau m_Q(\sqrt{\varepsilon} \partial_x)), \quad \Psi_{ \DL } (\sqrt{\varepsilon} \partial_x)= 1 +\mathcal{O}(\tau \DL(\sqrt{\varepsilon} \partial_x))
$$
yields that the remainder $\mathcal{R}_2(\tau,\varepsilon,u)$ is of the form
\begin{align*}
\mathcal{R}_2(\tau,\varepsilon,u) & = 
\mathcal{O}\left(c_L \tau^3 \frac{\varepsilon \partial_x}{1+ (\sqrt{\varepsilon}\vert \partial_x\vert)^{\beta_Q}}   \frac{\varepsilon^{4 } \partial_x^{10 }}{1+ (\sqrt{\varepsilon}\vert \partial_x\vert)^{2\beta_L}}  p(v)\right)\\& + \mathcal{O}\left(\tau^3 c_L \frac{\varepsilon \partial_x}{1+ (\sqrt{\varepsilon}\vert \partial_x\vert)^{\beta_Q}}   \frac{\varepsilon^{2 } \partial_x^{6 }}{1+ (\sqrt{\varepsilon}\vert \partial_x\vert)^{\beta_L}}  p(v) \right).
\end{align*}
This concludes the proof.
\end{proof}
Using Lemma \ref{lem:osc2} in the expansion of the exact solution \eqref{uexp2} yields together with 
\eqref{resi} and the definition of the numerical flow $\Phi^\tau$ in \eqref{schema2} that
\begin{equation}\label{expu2fin}
u(t_n+\tau) = \Phi^\tau(u(t_n)) +  \mathcal{R}_1(\tau,\varepsilon,u)+ \mathcal{R}_2(\tau,\varepsilon,u)
\end{equation}
where the remainders $ \mathcal{R}_1$ and  $ \mathcal{R}_2$ satisfy the bounds \eqref{2R1} and \eqref{2R2}, respectively.

\subsection{Error analysis}\label{sec:err2}

Let us again denote by $\varphi^t$ the exact flow of  \eqref{BBM} and by $\Phi^\tau$ the numerical flow defined by the scheme \eqref{schema2}, such that
$$
u(t_n+\tau) = \varphi^\tau(u(t_n) ) \quad \text{and}\quad u^{n+1} = \Phi^\tau(u^n).
$$

\subsubsection{Local error analysis}
\begin{lemma} \label{thm:loc2} Fix $r\geq 0$ and let $\beta_0 = \text{min}(1,\beta_Q)$,
$\beta_1 = \text{min}(2,\beta_Q+\beta_L)$ and $\beta_2 = \text{min}(6,\beta_Q+2 \beta_L)$.
Then, the local error  $ \varphi^\tau(u(t_n))- \Phi^\tau(u(t_n))$ satisfies 
\begin{align*}
\Vert & \varphi^\tau(u(t_n))- \Phi^\tau(u(t_n)) \Vert_r   \leq \tau^3 \varepsilon^2 c \left(\sup_{t_n \leq t \leq t_{n+1}}  \Vert  u\Vert_{r+5}\right)+ \tau^3 \varepsilon^2 \varepsilon^{1-\beta_0} c \left(\sup_{t_n \leq t \leq t_{n+1}}  \Vert  u\Vert_{r+7-2\beta_0}\right)\\
& + c_L\tau^3 \varepsilon^2 \varepsilon^{1-\beta_1/2} c \left(\sup_{t_n \leq t \leq t_{n+1}}  \Vert  u\Vert_{r+7-\beta_1}\right) +  c_L\tau^3   \varepsilon^2  \varepsilon^{3-\beta_2/2} c \left(\sup_{t_n \leq t \leq t_{n+1}}  \Vert  u\Vert_{r+11-\beta_2}\right) .
\end{align*}

\end{lemma}
\begin{proof}
The assertion follows from the expansion of the exact solution given in \eqref{expu2fin}
together with the error bounds on $ \mathcal{R}_1$ and  $ \mathcal{R}_2$ in \eqref{2R1} and \eqref{2R2}.
\end{proof}
\subsubsection{Stability analysis}
\begin{lemma} \label{thm:stab2} Fix  $r \geq 1$. The numerical flow $ \Phi^\tau$ defined by the scheme \eqref{schema2} is ${\varepsilon}$-stable in $H^r$ in the sense that for two functions $w \in H^{r+1}$ and $v \in H^r$ we have that 
$$
\Vert \Phi^\tau(w)- \Phi^\tau(v) \Vert_r  \leq  e^{\tau {\varepsilon} B} \Vert  w-v\Vert_r, \qquad B = B(\Vert w\Vert_{r+1}, \Vert v \Vert_r),
$$
where the constant $B$ depends on the $H^{r+1}$ norm of $w$ and  $H^r$ norm of $v$.
\end{lemma}
\begin{proof} Thanks to Lemma \ref{thm:stab2}  it remains to prove  the stability estimate only on the last three terms in \eqref{schema2}. The latter holds true thanks to the properties \eqref{filter} of the filter functions $\Psi_{m_Q}$ and $\Psi_{\DL}$.
\end{proof}

\subsubsection{Global error estimate}
\begin{theorem} \label{thm:glob2} Let $\beta_0 = \text{min}(1,\beta_Q)$,
$\beta_1 = \text{min}(2,\beta_Q+\beta_L)$ and $\beta_2 = \text{min}(6,\beta_Q+2 \beta_L)$. Then there exists a $\tau_0>0$ such that for all $0<\tau \leq \tau_0$  the global error estimate holds for  $u^n$ defined in \eqref{schema2}\begin{align*}
\Vert u(t_n)- u^n \Vert_{L^2} & \leq 
 \tau^2  t_n   \varepsilon^2  \left[ c \left(\sup_{t_n \leq t \leq t_{n+1}}  \Vert  u\Vert_{r+5}\right)+  \varepsilon^{1-\beta_0} c \left(\sup_{t_n \leq t \leq t_{n+1}}  \Vert  u\Vert_{r+7-2\beta_0}\right)\right.\\
& + c_L \left. \varepsilon^{1-\beta_1/2} c \left(\sup_{t_n \leq t \leq t_{n+1}}  \Vert  u\Vert_{r+7-\beta_1}\right) +  c_L  \varepsilon^{3-\beta_2/2} c \left(\sup_{t_n \leq t \leq t_{n+1}}  \Vert  u\Vert_{r+11-\beta_2}\right)\right]
e^{c_1 t_n  {\varepsilon}} ,
\end{align*}
where $c_1$ depends on the $H^2$ norm of   $u$.
\end{theorem}
\begin{proof}
The assertion follows by the local error bound given in Lemma \ref{thm:loc2} together with the stability estimate in Lemma \ref{thm:stab2} via a Lady Windermere's fan argument, see, e.g., \cite{HLW}.\end{proof}

\section{Numerical experiments}\label{sec:num} 
We underline our theoretical results with numerical experiments. As a model problem we take the BBM equation \eqref{bbm} and solve it with our first - and second-order long wave limit preserving schemes \eqref{scheme} and \eqref{schema2}, respectively,  for various values of $\varepsilon$ on  long time scales, i.e., up to $T= \frac{1}{\varepsilon}$. For the spatial discretisation we employ a standard pseudo spectral method. The numerical findings confirm the convergence order stated in Theorem \ref{thm:glob1} and Theorem \ref{thm:glob2}, respectively.

\begin{figure}[h!]
\centering
\includegraphics[width=0.471\linewidth]{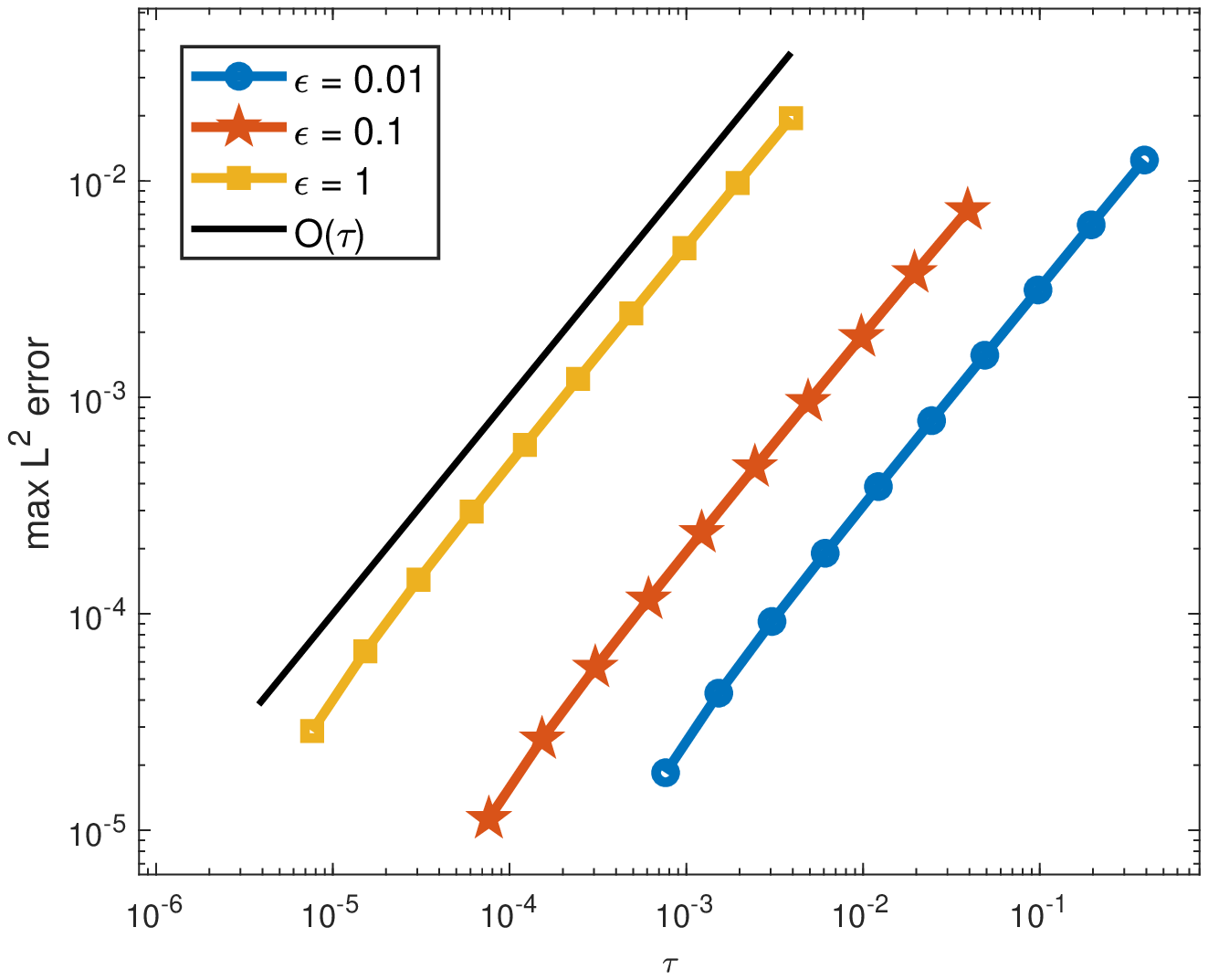}
\hfill
\includegraphics[width=0.471\linewidth]{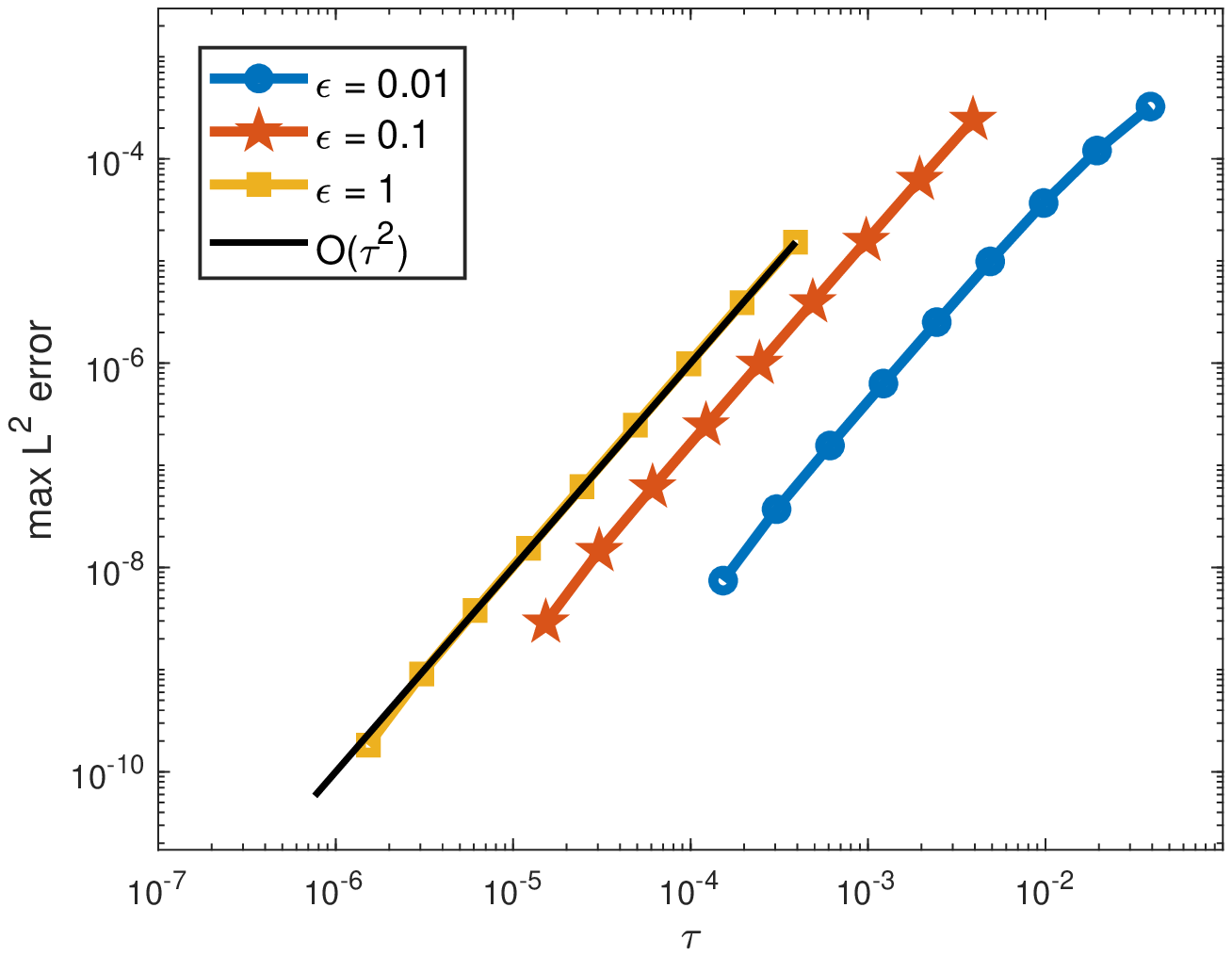}
\caption{Convergence plot ($L^2$ error versus step size) of the  first-order LWP scheme  \eqref{scheme}  (left) and the second-order LWP scheme  \eqref{schema2} (right) on  long time scales $t= \frac{1}{\varepsilon}$ for various values of $\varepsilon$. The black solid line corresponds to order one (left) and two (right), respectively. }\label{fig}
\end{figure}

\subsection*{Acknowledgements}

{\small
This project has received funding from the European Research Council (ERC) under the European Union’s Horizon 2020 research and innovation programme (grant agreement No. 850941).
}

\end{document}